\documentclass{amsart}
\newtheorem{theorem}{Theorem}[section]
\newtheorem{lemma}[theorem]{Lemma}
\newtheorem{corollary}[theorem]{Corollary}

\theoremstyle{remark}

\newtheorem{remark}[theorem]{Remark}

\title[On minimal non-$CL$-groups]{On minimal non-$CL$-groups}
\author[D.E. Otera]{Daniele Ettore Otera}

\address{D\'epartement de Math\'ematique\\
Batiment 425, Facult\'e de Science d'Orsay\\
Universit\'e Paris-Sud 11\\
F-91405, Orsay Cedex, France} 
\email{daniele.otera@math-psud.fr}

\author[F.G. Russo]{Francesco G. Russo}
\address{Mathematics Department\\
University of Palermo\\
via Archirafi 14\\
90123, Palermo, Italy}
\email{francescog.russo@yahoo.com}

\begin{document}


\begin{abstract}
If $m$ is a positive integer or infinity, the $m$-layer (or briefly, the layer) of a group $G$ is the subgroup $G_m$ generated by all elements of $G$ of order $m$. This notion goes back to some contributions of R. Baer and Ya.D. Polovickii of almost 60 years ago and is often investigated, because the presence of layers influences the group structure. If $G_m$ is finite for all $m$, $G$ is called  $FL$-group (or $FO$-group). A generalization is given by $CL$-groups, that is, groups in which $G_m$ is a Chernikov group for all $m$. By working on the notion of $CL$-group instead of that of $FL$-group, we extend a recent result of Z. Zhang, describing the structure of a group which is not a $CL$-group, but all whose proper subgroups are $CL$-groups.
\end{abstract}

\maketitle

\section{Introduction}

A group $G$ has
\textit{Chernikov conjugacy classes}, or briefly is a
 \textit{CC-group}, if $G/C_{G}(\langle x\rangle^G)$
is a Chernikov group for all $x\in
G$, where $C_G(\langle x \rangle^G)$ denotes the centralizer of the normal closure $\langle x \rangle^G$ of $\langle x \rangle$ in $G$. These groups were
introduced by Ya.D. Polovickii  in \cite{polovickii} and generalize the groups with \textit{finite conjugacy classes}, also known as \textit{FC-groups}. A classic reference is \cite{tomkinson} for the study of $FC$-groups. Among these groups, there is a special subclass,  which has received attention in \cite{zhang} and will be the subject of the present work.

From \cite[pp. 133--134]{robinson} we recall that, if $m$ is a positive integer or $\infty$, the \textit{m-layer} (or briefly, the \textit{layer}) of a group $G$ is the subgroup $G_m$ generated by all elements of $G$ of order $m$. If $G_m$ is finite for all $m$, $G$ is called  \textit{FL-group} (or \textit{FO-group}). An $FL$-group is characterized to have only a finite number of elements of each order, including $\infty$. This justifies the terminology $FO$-group, used by some authors. Of course, $FL$-groups are $FC$-groups, but their structure can be described more accurately with respect to that of $FC$-groups. In order to do this, we recall some notions from \cite{robinson}. Following \cite[p.135]{robinson}, a direct product of groups is called \textit{prime-thin} if for each prime $p$ at most a finite number of the direct factors contain elements of order $p$. Successively, we recall that a group $G$ is \textit{central-by-finite}, if its center $Z(G)$ has finite index in $G$. A group $G$ is said to be \textit{locally normal and finite} if each finite subset of $G$ is contained in a finite normal subgroup of $G$. Similarly, a group $G$ is said to be \textit{locally normal and Chernikov} if each finite subset of $G$ is contained in a Chernikov normal subgroup of $G$. Finally, we recall that a group $G$ satisfies min-$ab$, if it satisfies the minimal condition on its abelian subgroups.

Now we are able to state the main characterizations of $FL$-groups.

\begin{theorem}[See \cite{robinson}, Theorem 4.43]\label{t:1}
The following properties of a group $G$ are equivalent.
\begin{itemize}
\item[(i)] $G$ is an $FL$-group.
\item[(ii)]   $G$ is a locally normal and finite group and each Sylow subgroup satisfies min-$ab$.
\item[(iii)]   $G$ is isomorphic with a subgroup of a prime-thin direct product of central-by-finite Chernikov groups.
\end{itemize}
\end{theorem}

A group $G$ in which  $G_m$ is a Chernikov group for all $m$ is said to be a $CL$-group. Of course, $FL$-groups are $CL$-groups. There is not a rich  literature  in English language on $FL$-groups and $CL$-groups and \cite{bgm} may help the reader, who is interested to investigate the relations among ascending chains of $CC$-groups and the structure of $CL$-groups. However, weakening Theorem \ref{t:1}, $CL$-groups may be characterized analogously.

\begin{theorem}[See \cite{robinson}, Theorem 4.42]\label{t:2}
The following properties of a group $G$ are equivalent.
\begin{itemize}
\item[(i)] $G$ is a $CL$-group.
\item[(ii)]   $G$ is a locally normal and Chernikov group and each Sylow subgroup satisfies min-$ab$.
\item[(iii)]   $G$ is isomorphic with a subgroup of a prime-thin direct product of Chernikov groups.
\end{itemize}
\end{theorem}

If $\mathcal{X}$ is an arbitrary class of groups, $G$ is said to be
a \textit{minimal non}-$\mathcal{X}$-\textit{group}, or briefly an $MNX$-group, if it is not an
$\mathcal{X}$-group but all of whose proper subgroups are
$\mathcal{X}$-groups. Many results have been obtained on $MNX$-groups, for various choices of $\mathcal{X}$. 
If $\mathcal{X}$ is the class of $FC$-groups, we find the $MNFC$-groups characterized by V.V. Belyaev and N.F. Sesekin  in  \cite[Section
8]{tomkinson}.  They proved that an $MNFC$-group
is a finite cyclic extension of a divisible $p$-group of finite rank ($p$ a prime).
If $\mathcal{X}$ is the class of $CC$-groups, J. Ot\'{a}l and J. M. Pe\~{n}a proved in \cite[Theorem, p.1232]{otal-pena} that there are no $MNCC$-groups which have a non-trivial finite or abelian factor group. Similar subjects have been investigated in \cite{bruno-phillips, kp, leinen, russo-trabelsi, shum-zhang}.
More recently, Z. Zhang choose $\mathcal{X}$ to be the class of $FL$-groups,  proving in \cite[Theorem 2.5]{zhang} that all $MNFL$-groups are $MNFC$-groups. Consequently,  these groups may be described by the quoted classification of $MNFC$-groups.

In this paper we extend the results of Z. Zhang to $MNCL$-groups. In Section 2 we show that all $MNCL$-groups are $MNCC$-groups and this allows us to reduce the classification of $MNCL$-groups to that in \cite{otal-pena}. In Section 3 we characterize $MNCL$-groups and draw some conclusions on the perfect case.

\section{$MNCL$-groups}

An easy consequence of Theorem \ref{t:2} is listed below.

\begin{corollary}[See \cite{robinson}, p.134]\label{c:1}
$CL$-groups are countable and locally finite. 
\end{corollary}

The structure of a $CC$-group is well--known and described in \cite[Theorem 4.36]{robinson}. A consequence, which we will use in several arguments, is expressed below.

\begin{corollary}[See \cite{otal-pena}, p. 1234] \label{c:2}
The set of all elements of finite order of a $CC$-group $G$ is a locally normal and Chernikov characteristic subgroup of $G$. In particular, a periodic $CC$-group is a locally normal and Chernikov group. 
\end{corollary}

Torsion-free groups should be avoided in our investigations.

\begin{lemma}\label{l:1}
Let $G$ be an $MNCL$-group. Then $G$ is periodic.
\end{lemma}

\begin{proof}
 Assume that this is false and let $x$ be an element of infinite order. For any positive integer $n$, the subgroup $\langle x \rangle^n$ is a torsion-free proper subgroup of $G$. At the same time $\langle x \rangle^n$  is a $CL$-group and then it is periodic by Corollary \ref{c:1}. This contradiction implies the result.
\end{proof}

An important role is played by the normal subgroups whose factors are Chernikov groups. 

\begin{lemma}\label{l:2}
Let $G$ be  a $CC$-group and $H$ be a normal subgroup of $G$ such that $G/H$ is a Chernikov group. Then
$G$ is a $CL$-group if and only if $H$ is a $CL$-group.
\end{lemma}

\begin{proof}
If $G$ is a $CL$-group, then $H$ is of course a $CL$-group. Conversely, assume that $H$ is a $CL$-group. From Corollary \ref{c:1} $H$ is a periodic group, but also $G/H$ is a periodic group. Since the class of periodic groups is closed with respect to forming extensions of its members (see \cite[p.34]{robinson}), we conclude that $G$ is a periodic group. Now $G$ is a periodic $CC$-group and Corollary \ref{c:2} implies that $G$ is a locally normal and  Chernikov group. By Theorem \ref{t:2}, it remains to prove that each Sylow subgroup of $G$ satisfies min-$ab$. Let $P$ be a Sylow subgroup of $G$. $P\cap H$ is contained in  some Sylow $p$-subgroup of $H$, which is a $CL$-group and has all its Sylow subgroups satysfying min-$ab$ by Theorem \ref{t:2}. Therefore $P\cap H$ satisfies min-$ab$.  $P/(P \cap H)\simeq PH/H \le G/H$ is a Chernikov group and also satisfies min-$ab$. We conclude that $P$ is an extension of two groups with min-$ab$ and then it satisfies min-$ab$. The result follows.
\end{proof}

Also the subgroups of finite index play an important role.

\begin{lemma}\label{l:2bis}
Let $G$ be  a $CC$-group and $H$ be a subgroup of $G$ of finite index. Then
$G$ is a $CL$-group if and only if $H$ is a $CL$-group.
\end{lemma}

\begin{proof}
If $G$ is a $CL$-group, then $H$ is of course a $CL$-group. Conversely, assume that $H$ is a $CL$-group. From Corollary \ref{c:1} $H$ is a periodic group and so is $G$. Denoting with $H_G$ the core of $H$ in $G$, $|G:H_G|\le |G:H|$ is finite and then there is no loss of generality in assuming that $H$ is a normal subgroup of $G$.  Now $G$ is a periodic $CC$-group and Corollary \ref{c:2} implies that $G$ is a locally normal and  Chernikov group. By Theorem \ref{t:2}, it remains to prove that each Sylow subgroup of $G$ satisfies min-$ab$. Let $P$ be a Sylow subgroup of $G$. $P\cap H$ is contained in a some Sylow $p$-subgroup of $H$, which is a $CL$-group and has all its Sylow subgroups satysfying min-$ab$ by Theorem \ref{t:2}. Therefore $P\cap H$ satisfies min-$ab$. Since $|P:P \cap H|\le |PH:H| \le |G:H|$ is finite, we conclude that $P$ is a finite extension of a group with min-$ab$. Then it satisfies min-$ab$ and the result follows.
\end{proof}

The subgroups of $MNCL$-groups are subject of severe restrictions.

\begin{lemma}\label{l:3}
Let $G$ be a $CC$-group. If $G$ is an $MNCL$-group, then
there is no proper normal subgroup $H$ such that $G/H$ is a Chernikov group.
\end{lemma}

\begin{proof}
 Suppose that $H$ is a proper normal subgroup of $G$ such that $G/H$ is a Chernikov group. Then $H$ is a $CL$-group.  From Lemma \ref{l:2} $G$ is a $CL$-group, against the assumption. 
\end{proof}

\begin{lemma}\label{l:3bis}
Let $G$ be a $CC$-group. If $G$ is an $MNCL$-group, then
there is no proper subgroup $H$ of finite index.
\end{lemma}

\begin{proof}
 Suppose that $H$ is a proper subgroup of $G$ of finite index. Then $H$ is a $CL$-group.  From Lemma \ref{l:2bis} $G$ is a $CL$-group, against the assumption. 
\end{proof}

Now we prove the main result of the present section.

\begin{theorem}\label{t:3} All $MNCL$-groups are $MNCC$-groups.
\end{theorem}

\begin{proof}
Assume that $G$ is an $MNCL$-group. All proper subgroups of $G$ are $CL$-groups and then $CC$-groups. In order to complete the proof, it is enough to prove that $G$ is not a $CC$-group.

Assume that $G$ is a $CC$-group. For any element $x$ of $G$, the centralizer $C_G(\langle x\rangle^G)$ is a normal subgroup of $G$ such that $G/C_G(\langle x\rangle^G)$ is a Chernikov group. Therefore it must be trivial by Lemma \ref{l:3} and so $G=C_G(\langle x\rangle^G)$ for all $x$ in $G$. This means that $G$ is an abelian group. On another hand,  Lemma \ref{l:1} implies that $G$ is periodic, then $G$ is a periodic abelian group.

If $G$ is not an abelian $p$-group for some prime $p$, then each Sylow subgroup $P$ of $G$ is a proper subgroup of $G$ and hence a $CL$-group. Theorem \ref{t:2} implies that $P$ must be a Chernikov group. On another hand, $G$ is periodic abelian, then a locally normal and Chernikov group and by Theorem \ref{t:2} it should be a $CL$-group, which is a contradiction. 

Therefore we may assume that $G$ is an abelian $p$-group. However it cannot contain any proper subgroup of finite index by Lemma \ref{l:3bis}, hence it should be divisible, that is, the direct product of $m$ quasicyclic $p$-groups. If $m$ is finite, then $G$ is a Chernikov group, which is in contradiction with the fact that $G$ is not a $CL$-group.  If $m$ is infinite, then a proper subgroup $H$ of $G$ which is a direct product of an infinite number of quasicyclic $p$-groups cannot be a $CL$-group by Theorem \ref{t:2}. Then a proper subgroup $H$ of $G$ should be a direct product of a finite number of quasicyclic $p$-groups, then $H$ would be a Chernikov group, still against Lemma \ref{l:3}.

It follows that $G$ cannot be a $CC$-group, as claimed.
\end{proof}

\section{Consequences}

\cite[Theorem 2.3]{zhang} can be found as a special case of Theorem \ref{t:3}.
We need to recall that the \textit{finite residual} $G^*$ of a group $G$ is the intersection of all normal subgroups of $G$ of finite index. $G$ is said to be \textit{residually finite}, if $G^*$ is trivial. 

\begin{corollary}\label{c:3}  Let $G$ be a group in which the layers of the proper subgroups are of finite exponent. If $G$ is an $MNCL$-group, then $G$ is an $MNFC$-group.
\end{corollary}

\begin{proof} All $MNFL$-groups are $MNFC$-groups by \cite[Theorem 2.3]{zhang} and it is enough to prove that, if $G$ is an $MNCL$-group, then it is an $MNFL$-group.  

A proper subgroup $H$ of $G$ has its layers $H_m$ which are Chernikov groups of finite exponent. Then $H_m$ are finite groups for all $m\ge 1$. Consequently, $H$ is an $FL$-group. Since the choice of $H$ was aribitrary, the same is true for all proper subgroups of $G$ and then all proper subgroups of $G$ are $FL$-groups. On another hand, if $G$ is an $FL$-group, then it is a $CL$-group, against the assumption. Then $G$ is an $MNFL$-group, as claimed.
\end{proof}

Corollary \ref{c:3} allows us to apply the classification of V.V. Belyaev and N.F. Sesekin \cite[Theorem 8.13]{tomkinson}. This is shown in the next two results.

\begin{corollary}\label{c:4}
Assume that the layers of the proper subgroups of a group $G$ are  of finite exponent. $G$ is a nonperfect $MNCL$-group if and only if $G$ satisfies the following conditions:
\begin{itemize}
\item[(i)]   $G'=G^*$; $G=\langle G^*,x\rangle$, where $x^{p^n}\in G^*$, $x^p\in Z(G)$, $p$ is a prime and $n$ is a positive integer;
\item[(ii)]   $G^*$ can be expressed as a direct product of finitely many quasicyclic $q$-groups, where $q$ is a prime;
\item[(iii)]   There is no proper $G$-admissible subgroup in $G^*$.
\end{itemize}
\end{corollary}

\cite[Corollary 3.2]{zhang} shows the equivalence of the first four conditions in the next corollary and the fifth condition is due to Corollary \ref{c:4}. We should also mention that  the $MNFA$-groups and  the $MNCF$-groups, which we are going to characterize, are exactly the groups studied in \cite{shum-zhang}.

\begin{corollary} \label{c:5}
Assume that the layers of the proper subgroups of a nonperfect group $G$ are of finite exponent. Then the following conditions are equivalent:
\begin{itemize}
\item[(i)]   $G$ is an $MNFC$-group;
\item[(ii)]   $G$ is an $MNFA$-group;
\item[(iii)]   $G$ is an $MNCF$-group;
\item[(iv)]   $G$ is an $MNFL$-group;
\item[(v)]   $G$ is an $MNCL$-group.
\end{itemize}
\end{corollary}

We recall that a group $G$ is called \textit{locally graded} if every finitely generated subgroup of $G$ has a proper subgroup of finite index. As usual, the imposition of this condition is to avoid from our treatment the Tarski groups, that is, infinite non-abelian groups whose proper subgroups are finite.  For the case of $CC$-groups we know  as follows.

\begin{theorem}[See \cite{otal-pena}, Corollary, p.1232]\label{t:4}
If $G$ is a locally graded minimal non-$CC$-group, then $G$ is locally finite and countable. Furthermore, $G=G^*=G'$.  In particular, $G$ is perfect.
\end{theorem}

Therefore we may conclude the next result.

\begin{theorem}\label{t:5}
If $G$ is a locally graded $MNCL$-group, then $G$ is locally finite and countable. Furthermore, $G=G^*=G'$. In particular,   $G$ is perfect.
\end{theorem}

\begin{proof}
It is enough to combine Theorems \ref{t:3} and \ref{t:4}.
\end{proof}

The importance of Theorem \ref{t:5} is due to the fact that it describes a situation, which is completely different from that of $MNFL$-groups. In fact, \cite[Corollary 4.2]{zhang} states that there are no locally graded perfect $MNFL$-groups, while Theorem \ref{t:5} has just illustrated that all locally graded $MNCL$-groups are perfect. Two more properties of $MNCL$-groups are summarized below.

\begin{corollary}\label{c:6}
If $G$ is a locally graded $MNCL$-group, then
$G$ has no non-trivial finite factor groups. 
\end{corollary}

\begin{proof}
From Theorem \ref{t:5}, $G=G^*$ implies that there are no non-trivial finite factor groups.
\end{proof}

\begin{corollary}\label{c:7}
Assume that $G$ is an $MNCL$-group. If $G$ is  locally
graded, then $G$ is not finitely generated.
\end{corollary}

\begin{proof}  Assume that $G$ is a finitely generated locally graded $MNCL$-group. Then $G$ must have finite factor groups, against Theorem \ref{t:3}, which implies $G=G^*$. We conclude that $G$ cannot be finitely generated.  
\end{proof}

We end with two remarks which hide some deep open questions, related to the efforts in \cite{bruno-phillips,kp,leinen}.

\begin{remark}\label{r:1} From \cite[Remark]{zhang}, it is not known whether there exists a perfect 2-generated $MNFL$-group. Probably this is due to the fact that there are no examples of perfect $MNFL$-groups, which are not $MNFC$-groups.  Note that \cite[Theorem 4.1]{zhang} states that there are no locally finite perfect $MNFL$-groups. Then, if the desired examples exist, then they should be periodic but not locally finite. On another hand, the absence of such examples makes plausible that also the converse of \cite[Theorem 2.3]{zhang} would be true: one may expect that, not only all $MNFL$-groups are $MNFC$-groups, but that also the contrary is true. In fact, this is proved in the nonperfect case in \cite[Corollary 3.2]{zhang}.
\end{remark}

\begin{remark}\label{r:2} Similarly as in Remark \ref{r:1}, it is not known whether all $MNCC$-groups are $MNCL$-groups.
\end{remark}

\section*{Acknowledgements}
We are grateful to Professor Z. Zhang, who noted some weak points in the original version of the manuscript.

\end{document}